\documentclass[a4paper,11pt]{article}

\usepackage{amsmath,amssymb,amsthm}

\numberwithin{equation}{section}

\newtheorem{thm}{Theorem}[section]

\newtheorem{lem}[thm]{Lemma}
\newtheorem{cor}[thm]{Corollary}

\newtheorem*{rem}{Remark}

\theoremstyle{definition}

\allowdisplaybreaks

\begin{document}
\title{Remarks on the Rellich inequality}

\author{Shuji Machihara\,$^1$, Tohru Ozawa\,$^2$ and Hidemitsu Wadade\,$^3$}
\date{\small\it $^1$Department of Mathematics, Saitama University, 255 Shimookubo, Sakuraku, Saitama 338-8570 Japan (machihar@rimath.saitama-u.ac.jp)\vspace{.3cm}\\
$^2$Department of Applied Physics, Waseda University, Shinjuku, Tokyo 169-8555, Japan (txozawa@waseda.jp)\vspace{.3cm}\\
$^3$Faculty of Mechanical Engineering, Institute of Science and Engineering, Kanazawa University,\\
Kakuma, Kanazawa, Ishikawa 920-1192, Japan\\(wadade@se.kanazawa-u.ac.jp)}

\maketitle

\begin{abstract}

We study the Rellich inequalities in the framework of equalities. We
present equalities which imply the Rellich inequalities by dropping remainders. This provides a simple and
direct understanding of the Rellich inequalities as well as the nonexistence of nontrivial extremisers. 

\medskip

\noindent
2000 Mathematics Subject Classification\,:\,Primary 26D10, Secondary 46E35, 35A23.

\medskip

\noindent 
Key words\,:\,Rellich inequality, Hardy inequality. 
\end{abstract}

\section{Introduction and the main results}
\noindent

In this paper, we prove some equalities which yield the Rellich inequality 
by dropping remainder terms in $L^2(\Bbb R^n)$ setting for $n\geq 5$. 
Moreover, a characterization is given on $H^2$-functions which make vanishing remainders 
on the basis of simple partial differential equations. Our presentation based on 
equalities presumably gives a clear picture of how the Rellich inequality follows 
with sharp remainders and implies the nonexistence of nontrivial extremisers. 

\medskip

The Rellich inequality that we study in this paper is the following\,:
\begin{align}\label{Rellich}
&\left\|
\frac{f}{|x|^2}
\right\|_{L^2(\Bbb R^n)}
\leq\frac{4}{n(n-4)}\|\Delta f\|_{L^2(\Bbb R^n)}
\end{align}
for all $f\in H^2(\Bbb R^n)$ with $n\geq 5$, where $H^s=H^s(\Bbb R^n)$ 
is the standard Sobolev space of order $s\in\Bbb R$ defined as $(1-\Delta)^{-s/2}L^2(\Bbb R^n)$ and 
$\Delta=\sum_{j=1}^n \partial_j^2$ is the Laplacian in $\Bbb R^n$. 
The inequality \eqref{Rellich} is basic in the self-adjointness problem of 
the Schr\"odinger operators with singular potentials such as $V(x)=\lambda |x|^{-2}$ 
with $\lambda>-\frac{n(n-4)}{4}$ (See \cite{a-s, be, ca, da, e-l, g-g-m, gh-m, sch, t-z, ya} and references therein for related subjects). 
Moreover, there is a large literature on \eqref{Rellich} 
in connection with Hardy type inequalities \cite{a-c-r, BD, b-m, b-v, DV, FT, FS, g-p, gh-m1, h, II, IIO, k-w, m-o-w, m-o-w2, m-o-w3, m-o-w4, m, o-s, ta, z}. 

\medskip

In an earlier work \cite{m-o-w4}, we studied the Hardy inequality in $L^2$ setting 
by means of the equalities 
\begin{align}
\notag\left(\frac{n-2}{2}\right)^2\left\|
\frac{f}{|x|}
\right\|_{L^2(\Bbb R^n)}^2&=
\|\partial_r f\|_{L^2(\Bbb R^n)}^2-\|
|x|^{-\frac{n}{2}+1}\partial_r(|x|^{\frac{n}{2}-1}f)
\|_{L^2(\Bbb R^n)}^2\\
&\label{l2-earlier}=\|\partial_r f\|_{L^2(\Bbb R^n)}^2-\left\|
\partial_r f+\frac{n-2}{2|x|}f
\right\|_{L^2(\Bbb R^n)}^2
\end{align}
for all $f\in H^1(\Bbb R^n)$ with $n\geq 3$, where $\partial_r=\frac{x}{|x|}\cdot\nabla=\sum_{j=1}^n\frac{x_j}{|x|}\partial_j$ denotes the radial derivative in $\Bbb R^n$. 

\medskip

The purpose of this paper is to present the corresponding equalities on 
the Rellich inequality \eqref{Rellich} and characterize the equality case 
in terms of vanishing conditions of remainders. To be more specific, 
we prove the following theorem. 

\begin{thm}\label{thm1}
Let $n\geq 5$. Then the following equalities 
\begin{align}
&\notag\left(
\frac{n(n-4)}{4}
\right)^2
\left\|
\frac{f}{|x|^2}
\right\|_{L^2(\Bbb R^n)}^2\\
&\notag=\left\|
\partial_r^2 f+\frac{n-1}{|x|}\partial_r f
\right\|_{L^2(\Bbb R^n)}^2
-\left\|
\partial_r^2 f+\frac{n-1}{|x|}\partial_r f+\frac{n(n-4)}{4|x|^2}f
\right\|_{L^2(\Bbb R^n)}^2\\
&\notag\phantom{=}-\frac{n(n-4)}{2}\left\|
\frac{1}{|x|}\partial_r f+\frac{n-4}{2|x|^2}f
\right\|_{L^2(\Bbb R^n)}^2\\
&\notag=\left\|
|x|^{-n+1}\partial_r(|x|^{n-1}\partial_r f)
\right\|_{L^2(\Bbb R^n)}^2
-\left\|
|x|^{-\frac{n}{2}+1}\partial_r\left(
|x|^{-1}\partial_r(|x|^{\frac{n}{2}}f)
\right)
\right\|_{L^2(\Bbb R^n)}^2\\
&\notag\phantom{=}-\frac{n(n-4)}{2}\left\|
|x|^{-\frac{n}{2}+1}\partial_r(|x|^{\frac{n}{2}-2}f)
\right\|_{L^2(\Bbb R^n)}^2\\
&\notag=\left\|
|x|^{-n+1}\partial_r(|x|^{n-1}\partial_r f)
\right\|_{L^2(\Bbb R^n)}^2
-\left\|
|x|^{-\frac{n}{2}-1}\partial_r\left(
|x|^{3}\partial_r(|x|^{\frac{n-4}{2}}f)
\right)
\right\|_{L^2(\Bbb R^n)}^2\\
&\label{thm1-rellich}\phantom{=}-\frac{n(n-4)}{2}\left\|
|x|^{-\frac{n}{2}+1}\partial_r(|x|^{\frac{n}{2}-2}f)
\right\|_{L^2(\Bbb R^n)}^2
\end{align}
hold for all $f\in H^2(\Bbb R^n)$. Moreover, there does not exist $f\in H^2(\Bbb R^n)$ satisfying 
\begin{align}
\left(
\frac{n(n-4)}{4}
\right)^2\left\|
\frac{f}{|x|^2}
\right\|_{L^2(\Bbb R^n)}^2
&\notag=\left\|
\partial_r^2 f+\frac{n-1}{|x|}\partial_r f
\right\|_{L^2(\Bbb R^n)}^2\\
&\label{assume-eq}=\left\|
|x|^{-n+1}\partial_r(|x|^{n-1}\partial_r f)
\right\|_{L^2(\Bbb R^n)}^2
\end{align}
as well as 
\begin{align}\label{assume-eq2}
\left(\frac{n(n-4)}{4}\right)^2
\left\|
\frac{f}{|x|^2}
\right\|_{L^2(\Bbb R^n)}^2
=\|\Delta f\|_{L^2(\Bbb R^n)}^2
\end{align}
except $f=0$. 
\end{thm}

\medskip

Equalities \eqref{thm1-rellich} imply \eqref{Rellich} by the following theorem and its corollary. 
For $j$ with $1\leq j\leq n$, we denote by $L_j$ a spherical derivative defined by 
\begin{align*}
L_j=\partial_j-\frac{x_j}{|x|}\partial_r=\partial_j-\sum_{k=1}^n\frac{x_jx_k}{|x|^2}\partial_k. 
\end{align*}

\begin{thm}\label{thm2}
Let $n\geq 5$. Then the following equalities 
\begin{align}
\|\Delta f\|_{L^2(\Bbb R^n)}^2
&\notag=\left\|
\partial_r^2 f+\frac{n-1}{|x|}\partial_r f
\right\|_{L^2(\Bbb R^n)}^2
+\left\|
\sum_{j=1}^n L_j^2 f
\right\|_{L^2(\Bbb R^n)}^2\\
&\notag\phantom{=}
+\frac{n(n-4)}{2}\sum_{j=1}^n\left\|
\frac{1}{|x|}L_j f
\right\|_{L^2(\Bbb R^n)}^2
+2\sum_{j=1}^n\left\|
\partial_r L_jf+\frac{n-2}{2|x|}L_j f
\right\|_{L^2(\Bbb R^n)}^2\\
&\notag=\|
|x|^{-n+1}\partial_r(|x|^{n-1}\partial_r f)
\|_{L^2(\Bbb R^n)}^2
+\left\|
\sum_{j=1}^n L_j^2 f
\right\|_{L^2(\Bbb R^n)}^2\\
&\label{thm2-state}\phantom{=}+\frac{n(n-4)}{2}\sum_{j=1}^n\left\|\frac{1}{|x|}L_j f\right\|_{L^2(\Bbb R^n)}^2
+2\sum_{j=1}^n\|
|x|^{-\frac{n}{2}+1}\partial_r(|x|^{\frac{n}{2}-1}L_j f)
\|_{L^2(\Bbb R^n)}^2
\end{align}
hold for all $f\in H^2(\Bbb R^n)$. 
\end{thm}

\begin{cor}\label{cor-rellich}
Let $n\geq 5$. Then the inequality 
\begin{align}\label{rellich-cor}
\|\Delta f\|_{L^2(\Bbb R^n)}^2\geq \left\|\partial_r^2 f+\frac{n-1}{|x|}\partial_r f\right\|_{L^2(\Bbb R^n)}^2
\end{align}
holds for all $f\in H^2(\Bbb R^n)$. In \eqref{rellich-cor}, equality holds if and only if $f$ is radial. 
\end{cor}

We prove Theorems \ref{thm1} and \ref{thm2} in Sections 2 and 3, respectively. 
For simplicity, we prove the theorems for $f\in C^\infty_0(\,\Bbb R^n\setminus\{0\}\,;\Bbb C\,)$ since 
the proofs are completed by a density argument. The main idea of the proofs is given by the following lemma. 

\begin{lem}\label{lem-bector}
Let ${\cal H}$ be a vector space with Hermitian scalar product $(\cdot|\cdot)$. 
Also let $a\in\Bbb R$, $c>0$ and $u,v\in{\cal H}$. Then the following equalities are equivalent. 
\begin{align*}
&\|u\|^2=-c\operatorname{Re}(u|v)+a.\\
&\operatorname{Re}(u|u+cv)=a.\\
&\|cv\|^2=\|u+cv\|^2+\|u\|^2-2a.\\
&\frac{1}{c^2}\|u\|^2=\|v\|^2-\left\|v+\frac{1}{c}u\right\|^2+\frac{2a}{c^2}. 
\end{align*}
\end{lem}

\begin{proof}
The lemma follows from the equality 
\begin{align*}
\|cv\|^2=\|u+cv\|^2+\|u\|^2-2\operatorname{Re}(u|u+cv). 
\end{align*}
\end{proof}

\begin{rem}
The lemma was first formulated in \cite{m-o-w4} for $a=0$. 
In \cite{m-o-w4}, the equalities \eqref{l2-earlier} were derived from 
\begin{align*}
\int_{\Bbb R^n}\frac{|f(x)|^2}{|x|^2}dx
=-\frac{2}{n-2}\operatorname{Re}\int_{\Bbb R^n}\frac{f(x)}{|x|}\overline{\partial_r f(x)}dx 
\end{align*}
by applying the lemma with ${\cal H}=L^2(\Bbb R^n)$, $u=\frac{f}{|x|}$, $v=\partial_r f$, and $c=\frac{2}{n-2}$. 
\end{rem}

\section{Proof of Theorem \ref{thm1}}
\noindent

We introduce the standard polar coordinates $(r,\omega)=\left(|x|,\frac{x}{|x|}\right)\in(0,\infty)\times \Bbb S^{n-1}$ 
and the Lebesgue measure $\sigma$ on the unit sphere $\Bbb S^{n-1}$. 
We have by integration by parts 
\begin{align}
&\notag\int_{\Bbb R^n}\frac{|f(x)|^2}{|x|^4}dx\\
&\notag=\int_0^\infty r^{n-5}\int_{\Bbb S^{n-1}}|f(r\omega)|^2d\sigma(\omega)dr\\
&\notag=-\frac{2}{n-4}\operatorname{Re}\int_0^\infty r^{n-4}\int_{\Bbb S^{n-1}}f(r\omega)\overline{\omega\cdot\nabla f(r\omega)}d\sigma(\omega)dr\\
&\notag=\frac{2}{(n-3)(n-4)}\operatorname{Re}\int_0^\infty r^{n-3}\int_{\Bbb S^{n-1}}
\left(
|\omega\cdot\nabla f(r\omega)|^2
+f(r\omega)\overline{(\omega\cdot\nabla)^2 f(r\omega)}
\right)d\sigma(\omega)dr\\
&\label{first-by-parts}=\frac{2}{(n-3)(n-4)}
\left(
\left\|
\frac{1}{|x|}\partial_r f
\right\|_{L^2(\Bbb R^n)}^2
+\operatorname{Re}\int_{\Bbb R^n}\frac{f(x)}{|x|^2}\,\overline{\partial_r^2 f(x)}dx
\right).
\end{align}
The first norm on the right hand of the last equality in \eqref{first-by-parts} is rewritten as 
\begin{align}
\left\|\frac{1}{|x|}\partial_r f\right\|_{L^2(\Bbb R^n)}^2
&\notag=\left\|\partial_r\left(\frac{f}{|x|}\right)+\frac{f}{|x|^2}\right\|_{L^2(\Bbb R^n)}^2\\
&\label{eq1-int}=\left\|\partial_r\left(\frac{f}{|x|}\right)\right\|_{L^2(\Bbb R^n)}^2+
2\operatorname{Re}\left(\partial_r\left(\frac{f}{|x|}\right)\Bigg|\frac{f}{|x|^2}\right)
+\left\|\frac{f}{|x|^2}\right\|_{L^2(\Bbb R^n)}^2. 
\end{align}
We apply \eqref{l2-earlier} with $f$ replaced by $\frac{f}{|x|}$ to obtain 
\begin{align}
\label{eq2-int}\left\|
\partial_r\left(\frac{f}{|x|}\right)
\right\|_{L^2(\Bbb R^n)}^2
=\left(\frac{n-2}{2}\right)^2\left\|\frac{f}{|x|^2}\right\|_{L^2(\Bbb R^n)}^2
+\left\|
|x|^{-\frac{n}{2}+1}\partial_r(|x|^{\frac{n}{2}-2}f)
\right\|_{L^2(\Bbb R^n)}^2. 
\end{align}
Integrating by parts, we have 
\begin{align}
2\operatorname{Re}\left(
\partial_r\left(\frac{f}{|x|}\right)\Bigg|\frac{f}{|x|^2}
\right)&\notag=\int_{\Bbb R^n}\frac{1}{|x|}\partial_r\left(\frac{|f|^2}{|x|^2}\right)dx
=\int_{\Bbb R^n}\frac{x}{|x|^2}\cdot\nabla\left(\frac{|f|^2}{|x|^2}\right)dx\\
&\label{eq3-int}=-(n-2)\left\|\frac{f}{|x|^2}\right\|_{L^2(\Bbb R^n)}^2. 
\end{align}
By \eqref{eq2-int} and \eqref{eq3-int}, we rewrite \eqref{eq1-int} as 
\begin{align}
&\left\|
\frac{1}{|x|}\partial_r f
\right\|_{L^2(\Bbb R^n)}^2
=\left(\frac{n-4}{2}\right)^2
\left\|\frac{f}{|x|^2}\right\|_{L^2(\Bbb R^n)}^2
+\left\|
|x|^{-\frac{n}{2}+1}\partial_r(|x|^{\frac{n}{2}-2}f)
\right\|_{L^2(\Bbb R^n)}^2.\label{eq4-int}
\end{align}
The second integral on the right hand side of the last equality in \eqref{first-by-parts} is rewritten as 
\begin{align}
&\notag\operatorname{Re}\int_{\Bbb R^n}\frac{f(x)}{|x|^2}\overline{\partial_r^2 f(x)}dx\\
&=\operatorname{Re}\int_{\Bbb R^n}\frac{f(x)}{|x|^2}\overline{
\left(
\partial_r^2f(x)+\frac{n-1}{|x|}\partial_r f(x)
\right)
}dx-(n-1)\operatorname{Re}\int_{\Bbb R^n}\frac{f(x)}{|x|^3}\overline{\partial_r f(x)}dx,\label{eq5-int}
\end{align}
where the last integral is given by 
\begin{align}
&\operatorname{Re}\int_{\Bbb R^n}\frac{f}{|x|^3}\overline{\partial_r f}dx
=\frac{1}{2}\int_{\Bbb R^n}\frac{1}{|x|^3}\partial_r(|f|^2)dx=\frac{1}{2}\int_{\Bbb R^n}\frac{x}{|x|^4}\cdot\nabla(|f|^2)dx
=-\frac{n-4}{2}\left\|\frac{f}{|x|^2}\right\|_{L^2(\Bbb R^n)}^2. \label{eq6-int}
\end{align}
It follows from \eqref{first-by-parts}, \eqref{eq4-int}, \eqref{eq5-int} and \eqref{eq6-int} that 
\begin{align}
&\notag\frac{n(n-4)}{4}\left\|\frac{f}{|x|^2}\right\|_{L^2(\Bbb R^n)}^2\\
&=-\operatorname{Re}\int_{\Bbb R^n}\frac{f(x)}{|x|^2}
\overline{
\left(
\partial_r^2 f(x)+\frac{n-1}{|x|}\partial_r f(x)
\right)
}dx-\left\|
|x|^{-\frac{n}{2}+1}\partial_r(|x|^{\frac{n}{2}-2}f)
\right\|_{L^2(\Bbb R^n)}^2.\label{eq7-int}
\end{align}
Then \eqref{thm1-rellich} follows from \eqref{eq7-int} by applying Lemma \ref{lem-bector} with ${\cal H}=L^2(\Bbb R^n)$, 
$u=\frac{f}{|x|^2}$, $v=\partial_r^2 f+\frac{n-1}{|x|}\partial_r f$, $c=\frac{4}{n(n-4)}$ and 
$
a=-c\|
|x|^{-\frac{n}{2}+1}\partial_r(|x|^{\frac{n}{2}-2}f)
\|_{L^2(\Bbb R^n)}^2
$. 

\medskip

We now assume that $f\in H^2(\Bbb R^n)$ satisfies \eqref{assume-eq}. 
Then by \eqref{thm1-rellich}, it follows $\partial_r(|x|^{\frac{n}{2}-2}f)=0$, which is equivalent to the existence of $\varphi:{\Bbb S}^{n-1}\to\Bbb C$ 
such that $|x|^{\frac{n}{2}-2}f(x)=\varphi(\frac{x}{|x|})$ almost everywhere. 
In that case $\frac{f}{|x|^2}\in L^2(\Bbb R^n)$ if and only if $\frac{1}{|x|^n}|\varphi(\frac{x}{|x|})|^2\in L^1(\Bbb R^n)$, 
where the last condition if and only if $\varphi\equiv 0$, which in turn implies $f\equiv 0$. 
In the case where $f\in H^2(\Bbb R^n)$ satisfies \eqref{assume-eq2}, 
where the problem is reduced to the case \eqref{assume-eq} just we have argued if we can prove \eqref{rellich-cor}. 
Therefore, the proof of the last part of the theorem will be completed after the completion of the proof of Corollary \ref{cor-rellich}. 

\section{Proof of Theorem \ref{thm2}}
\noindent

We start with the equality 
\begin{align*}
\Delta f=\partial_r^2 f+\frac{n-1}{|x|}\partial_r f+\sum_{j=1}^n L_j^2 f, 
\end{align*}
which is verified by a direct calculation. Then we expand the scalar product as 
\begin{align}
\label{thm2-prf1}\|\Delta f\|_{L^2(\Bbb R^n)}^2
=\left\|
\partial_r^2 f+\frac{n-1}{|x|}\partial_r f
\right\|_{L^2(\Bbb R^n)}^2
+\left\|\sum_{j=1}^n L_j^2 f\right\|_{L^2(\Bbb R^n)}^2
+2\operatorname{Re}\left(
\partial_r^2f+\frac{n-1}{|x|}\partial_r f\Bigg|\sum_{j=1}^n L_j^2 f
\right).  
\end{align}
From now on we consider the last scalar product. For simplicity, let 
\begin{align*}
&g=\partial_r^2 f+\frac{n-1}{|x|}\partial_r f=|x|^{-n+1}\partial_r(|x|^{n-1}\partial_r f)
\text{ \,and \,}h_j=L_j f. 
\end{align*}
By integration by parts, 
\begin{align*}
(g|L_j h_j)=-(L_jg|h_j)+(n-1)\left(g\Big|\frac{x_j}{|x|^2}h_j\right). 
\end{align*}
This gives 
\begin{align}\label{exchange-2}
\left(g\Bigg|\sum_{j=1}^n L_j^2 f\right)=-\sum_{j=1}^n(L_j g|h_j)
\end{align}
since $\sum_{j=1}^n x_jL_j=0$. We also notice that $L_j\partial_r=\left(\partial_r+\frac{1}{|x|}\right)L_j$ 
and that $L_j(|x|^\lambda u)=|x|^\lambda L_j u$ for any $\lambda\in\Bbb R$ to obtain 
\begin{align}
L_jg&\notag=L_j(
|x|^{-n+1}\partial_r(|x|^{n-1}\partial_r f)
)=|x|^{-n+1}L_j\partial_r(|x|^{n-1}\partial_r f)\\
&\notag=|x|^{-n+1}\left(\partial_r+\frac{1}{|x|}\right)L_j(|x|^{n-1}\partial_r f)
=|x|^{-n+1}\left(\partial_r+\frac{1}{|x|}\right)|x|^{n-1}L_j\partial_r f\\
&\label{exchange-1}=|x|^{-n+1}\left(\partial_r+\frac{1}{|x|}\right)|x|^{n-1}\left(\partial_r+\frac{1}{|x|}\right)h_j
=\partial_r^2h_j+\frac{n+1}{|x|}\partial_r h_j+\frac{n-1}{|x|^2}h_j. 
\end{align}
By \eqref{exchange-1} and \eqref{l2-earlier}, the real part of the left hand side of \eqref{exchange-2} is calculated as 
\begin{align}
-\operatorname{Re}\sum_{j=1}^n(L_jg|h_j)
&\notag=-\sum_{j=1}^n\operatorname{Re}\left(
(\partial_r^2 h_j|h_j)+(n+1)\operatorname{Re}\left(\frac{1}{|x|}\partial_r h_j\Bigg|h_j\right)
+(n-1)\left\|\frac{1}{|x|}h_j\right\|_{L^2(\Bbb R^n)}^2
\right)\\
&\notag=-\sum_{j=1}^n\left(
\frac{(n-1)(n-2)}{2}\left\|
\frac{1}{|x|}h_j
\right\|_{L^2(\Bbb R^n)}^2-\|\partial_r h_j\|_{L^2(\Bbb R^n)}^2\right.\\
&\notag\phantom{=}\left.-\frac{(n+1)(n-2)}{2}\left\|
\frac{1}{|x|}h_j
\right\|_{L^2(\Bbb R^n)}^2
+(n-1)\left\|\frac{1}{|x|}h_j\right\|_{L^2(\Bbb R^n)}^2
\right)\\
&\notag=\sum_{j=1}^n\left(
\|\partial_r h_j\|_{L^2(\Bbb R^n)}^2-\left\|\frac{1}{|x|}h_j\right\|_{L^2(\Bbb R^n)}^2
\right)\\
&\label{thm2-prf2}=\sum_{j=1}^n\left(
\frac{n(n-4)}{4}\left\|\frac{1}{|x|}h_j\right\|_{L^2(\Bbb R^n)}^2
+\||x|^{-\frac{n}{2}+1}
\partial_r(|x|^{\frac{n}{2}-1}h_j)
\|_{L^2(\Bbb R^n)}^2
\right).
\end{align}
By \eqref{thm2-prf1}, \eqref{exchange-2} and \eqref{thm2-prf2}, we obtain \eqref{thm2-state}. 

\begin{proof}[Proof of Corollary \ref{cor-rellich}]
The inequality \eqref{rellich-cor} follows immediately from \eqref{thm2-state}. In \eqref{rellich-cor}, 
equality holds only if $\sum_{j=1}^n L_j^2 f=0$, which is equivalent to the fact that $f$ is radial since 
\begin{align*}
\frac{1}{|x|^2}\sum_{1\leq j<k\leq n}(x_j\partial_k-x_k\partial_j)^2f=\sum_{j=1}^n L_j^2 f. 
\end{align*}
Conversely, if $f$ is radial, then $L_jf=0$ for all $j$ and \eqref{rellich-cor} is realized as an equality. 
\end{proof}

\end{document}